\documentclass[12pt]{article}

\usepackage[hmargin=1in,vmargin=1in]{geometry} 

\usepackage{verbatim}
\usepackage{amsmath}
\usepackage{amssymb}
\usepackage{amsthm}
\usepackage{tikz}
\usepackage{url}
\usepackage{graphicx}
\usepackage{mleftright}
\usepackage{color,hyperref}
\usepackage{booktabs}
\definecolor{darkblue}{rgb}{0.0,0.0,0.3}
\hypersetup{colorlinks,breaklinks,
            linkcolor=darkblue,urlcolor=darkblue,
            anchorcolor=darkblue,citecolor=darkblue}

\newtheorem{thm}{Theorem}
\newtheorem{prop}[thm]{Proposition}
\newtheorem{lem}[thm]{Lemma}
\newtheorem{coro}[thm]{Corollary}
\newtheorem{conj}[thm]{Conjecture}

\newtheorem{prob}[thm]{Problem}
\newtheorem{defi}[thm]{Definition}

\newtheorem{ex}[thm]{Example}

\newtheorem{rem}[thm]{Remark}

\newcommand{\R}{\mathbb{R}}

\providecommand{\keywords}[1]{\textbf{\textit{keywords:}} #1}

\title {Quotients of uniform positroids}

\author{Carolina Benedetti\\[-5pt]
\small Departamento de Matem\'aticas, Universidad de los Andes, Bogot\'a, Colombia.\\{\tt \small c.benedetti@uniandes.edu.co}
\\[5pt]
Anastasia Chavez\\[-5pt]
\small Department of Mathematics, UC Davis, USA.\\{\tt \small anachavez@math.ucdavis.edu}
\\[5pt]
Daniel Tamayo\\[-5pt]
\small LRI, Univ. Paris-Sud - Univ. Paris-Saclay, Orsay, France.\\{\tt \small daniel.tamayo-jimenez@lri.fr}
}

\date{}

\begin{document}


\maketitle


\begin{abstract}

Flag matroids are a rich family of Coxeter matroids that can be characterized using pairs of matroids that form a quotient. We consider a class of matroids called positroids, introduced by Postnikov, and utilize their combinatorial representations to explore characterizations of flag positroids. 

Given a uniform positroid, we give a purely combinatorial characterization of a family of positroids that form quotients with it. We state this in terms of their associated decorated permutations. In proving our characterization we also fully describe the circuits of this family.

\end{abstract}

\keywords
{Positroids, quotients of matroids, decorated permutations.}

\section{Introduction}

Flag matroids, introduced in~\cite{Cox, GS87, GS87Greed}, are a rich family of Coxeter matroids which can be characterized in several ways. If two matroids $M_1$ and $M_2$ on the same ground set form a flag and their ranks satisfy $r_{M_1}<r_{M_2}$ the we say that $M_1$ is a quotient of $M_2$, or that $M_1$ and $M_2$ are concordant. The main contribution of this paper is to exploit the combinatorics of a special family of matroids called positroids in order to characterize when certain pairs of positroids form a quotient, and thus form a flag positroid.

Introduced in \cite{pos}, positroids have proven to be a combinatorially exciting family of matroids. Positroids have rich connections to total positivity \cite{kodamalauren}, cluster algebras \cite{clusteralgeb}, and physics \cite{agarwala1}. There are several ways to describe positroids combinatorially~\cite{oh} either via Grassmann necklaces, decorated permutations, Le-diagrams, among other combinatorial objects as stated in~\cite{pos}. With such rich combinatorics, one can ask if certain matroidal properties may be better understood in the case of positroids through any of these objects.

It is known that the uniform matroid $U_{k,n}$ is a positroid. We describe combinatorially positroids of rank $k-1$ that are a quotient of $U_{k,n}$. Our characterization is a complete one for $n \leq 6$. The combinatorial description made here is done by providing the decorated permutation associated to the given quotients of $U_{k,n}$.
In this way we then obtain a partial answer to the problem stated in~\cite{oh}, namely, determine combinatorially when two positroids are concordant. Our work also includes a concrete description of the circuits of those positroids concordant to $U_{k,n}$.

On the other hand, our results provide a way of better understanding some flag positroids. This opens the door to determine the realizability of positively oriented flag matroids, in the spirit of~\cite{ARWOriented}, where the authors prove that positively oriented matroids are realizable. 

In order to understand quotients of positroids, we introduce the \emph{poset of positroid quotients}, whose elements are positroids on the same ground set and whose covering relation is given by $N\lessdot M$ if and only if $N$ is a quotient of $M$ and their ranks differ by one.  

We conclude with a conjecture establishing a necessary and sufficient condition for two arbitrary positroids to form a quotient. This conjecture is stated in terms of decorated permutations as well.

The paper is organized as follows: in Section \ref{prelim} we provide the necessary background on positroids and quotients of matroids. In Section \ref{poset:sec} we introduce the poset of quotients of positroids and explore some of its combinatorics. We also characterize families of positroids that are quotients of uniform positroids, and conjecture a general combinatorial rule for positroid quotients. In Section \ref{future} we end with current work and some further questions.

\section{Preliminaries}\label{prelim}

Matroids are combinatorial objects that generalize the notion of linear independence. There are several equivalent ways to define them but we will focus only on the basis definition of a matroid. We suggest \cite{Ox} for a wider view on matroid theory.

Throughout this work we denote by $[n]$ the set $\{1,2,\dots,n\}$ and the $k$-subsets of $[n]$ by ${\mleft(\genfrac..{0pt}{}{[n]}{k}\mright)}$. Whenever there is no room for confusion, we will denote the set $\{a_1,\ldots,a_n\}$ by $a_1\ldots a_n$.

\subsection{Matroids}

\begin{defi}
    A \emph{matroid} M is an ordered pair $(E,\mathcal{B})$ that consists of a finite set $E$ and a collection $\mathcal{B}$ of subsets of $E$ that satisfies the following conditions:
    \begin{enumerate}
        \item[(B1)] $\mathcal{B}\not=\emptyset$,
        
        \item[(B2)] If $B_1,B_2$ are distinct elements in $\mathcal{B}$ and $x\in B_1\setminus B_2$, then there exists an element $y\in B_2\setminus B_1$ such that $(B_1\setminus\{x\})\cup\{y\}\in\mathcal{B}$.
    \end{enumerate} The set $E$ is the {ground set} of $M$ and the collection $\mathcal{B}:=\mathcal{B}(M)$ is called the \emph{set of bases} of $M$.
\end{defi}

It can be shown that every element of $\mathcal{B}$ has the same cardinality, denoted $r_M$, and is called the \emph{rank} of $M$. We will say that a subset $I$ of $E$ is \emph{independent} in the matroid $M=(E,\mathcal B)$ if there exists $B\in\mathcal B$ such that $I\subseteq B$. In particular, notice that $\emptyset$ is always independent. We denote by $\mathcal{I}(M)$ the collection of independent sets of a matroid $M$. If $I$ is not independent we say it is \emph{dependent}. In particular, a minimally dependent subset $C$ of $E$ is called a \emph{circuit} of $M$. That is, $C$ is dependent in $M$ but every proper subset of $C$ is independent. We denote by $\mathcal{C}(M)$ the collection of circuits of the matroid $M$.

A classic example of a matroid (and one of the principal objects in our work) is the uniform matroid.

\begin{defi}
    Let $n$ be a positive integer and $0\leq k\leq n$. Denote by $U_{k,n}$ the ordered pair $\left([n],{\mleft(\genfrac..{0pt}{}{[n]}{k}\mright)}\right)$. 
\end{defi}

\begin{thm}
    $U_{k,n}=\left([n],{\mleft(\genfrac..{0pt}{}{[n]}{k}\mright)}\right)$ is a matroid of rank $k$ and it is called the \emph{uniform matroid} on $[n]$ of rank $k$.
\end{thm}

There are several operations on matroids but we will only make use of the following.

\begin{defi}
The \emph{dual} of a matroid $M=(E,\mathcal{B})$ is the ordered pair $M^*=(E,\mathcal{B}^*)$ where $$B^*=\{E\setminus B \,:\, B\in\mathcal{B}\}.$$
\end{defi}

One can easily see that the dual of a matroid is a matroid as well and if $r_M=k$, then $r_{M^*}=n-k$. Thus one obtains that, $U_{k,n}^*=U_{n-k,n}$.


\subsection{Quotients and flag matroids}

In this paper we are concerned with \emph{quotients} of a particular class of matroids that will be defined in Section \ref{sec:positroid}. Thus we now recall some matroid theory on quotients.

\begin{defi}\label{def:quotient}
Given two matroids $M$ and $N$ on the same ground set $E$, we say that $M$ is a \emph{quotient} of $N$ if every circuit of $N$ can be written as the union of circuits of $M$. 
\end{defi}

\begin{ex}
The matroid $U_{1,3}$ is a quotient of $U_{2,3}$. This is clear since the only circuit of $U_{2,3}$ is the set $\{1,2,3\}$. Which can be written as the union of $\{1,2\},\{2,3\},\{1,3\}$, which are the circuits of $U_{1,3}$. In general, $U_{k,n}$ is a quotient of $U_{\ell,n}$ as long as $k\leq \ell$.
\end{ex}

Definition \ref{def:quotient} has been studied in in the context of \emph{strong maps} in Chapter 8 of \cite{white}. In fact, quotients can be defined in many equivalent ways. As such we present the following proposition whose proof we omit and can be found in \cite[Prop. 8.1.6]{white}. 

\begin{prop}\label{prop:strongmaps}
Let $M$ and $N$ be matroids on the same ground set $E$. The following statements are equivalent:
\begin{itemize}
    \item[(a)] $M$ is a quotient of $N$.
    \item[(b)] $N^*$ is a quotient of $M^*$.
    \item[(c)] For any pair of subsets $A$ and $B$ of $E$, such that $A\subset B$, it follows that
    $$r_N(B)-r_N(A)\geq r_M(B)-r_M(A).
    $$
\end{itemize}{}
\end{prop}

Notice that by part (c) of Proposition
\ref{prop:strongmaps} if we take $A=\emptyset$ it follows that $r_N(B)\geq r_M(B)$ whenever $M$ is a quotient of $N$. Moreover, equality holds in this case for any $B$ if and only if $M=N$. In view of this, the following definition is in order.

\begin{defi}\label{def:flagmatroid}
Let $M_1,\dots,M_k$ be a collection of distinct matroids on the ground set $E$. If for every $1\leq i<j\leq k$ it holds that $M_i$ is a quotient of $M_j$ then we say that the collection $\{M_1,\dots,M_k\}$ is a \emph{flag matroid}. We denote this as $M_1\subset\cdots\subset M_k$ and we refer to the matroids $M_i$ as the \emph{constituents} of the flag matroid. 

If $M_1\subset\cdots\subset M_k$ is a flag matroid, we sometimes refer to its constituents as being \emph{concordant}. That is, a collection of matroids $\{M_1\dots,M_k\}$ is concordant if $M_i$ is a quotient of $M_j$, or vice versa, for all pairs of distinct indexes $i,j$. This property as a whole is called \emph{concordance}.
\end{defi}

If $M$ and $N$ are matroids on the ground set $E$ and if $M$ is a quotient of $N$ then every basis of $M$ is contained in a basis of $N$. Reciprocally, every basis of $N$ contains a basis of ~$M$ \cite{BorovikGelfantWhiteFlags}. Now, if  $\{M_1,\dots,M_k\}$ is a collection of matroids of ranks $r_1<\cdots<r_k$, respectively, then in order for this collection to form a flag matroid, it suffices to check that $M_i$ is a quotient of $M_{i+1}$, for $i=1,\dots,k-1$. Thus the terminology \emph{flag matroid} is natural as the collection of bases in each $\mathcal B(M_i)$ form sequences, or flags, of the form $B_1\subset\cdots\subset B_k$, where $B_i\in\mathcal B(M_i)$.
When $k=n=|E|$ and the sequence of ranks are such that $r_i=i$, for $i\in[n]$, we say that the collection 
$\{M_1,\dots,M_n\}$ is a \emph{full flag matroid}.

\begin{ex}
Let $M_1$ be the matroid on the ground set $[4]$ whose base set is $\mathcal{B}(M_1)=\{\{1\},\{3\},\{4\}\}$. Let $M_2=U_{3,4}$. Then $\{M_1,M_2\}$ is a flag matroid since the circuit $\{1,2,3,4\}$ of $M_2$ can be written as union of the sets $\{2\}, \{1,3\},\{1,4\}$, which are circuits of $M_1$. Also, the bases of $M_1$ and $M_2$ form the following 9 flags \begin{equation*}
    \begin{array}{ccc}
         \{1\}\subset \{1,2,3\} \quad &\quad \{1\}\subset \{1,2,4\}\quad & \quad\{1\}\subset \{1,3,4\}  \\
         \{3\}\subset \{1,2,3\}\quad &\quad \{3\}\subset \{1,3,4\}\quad & \quad\{3\}\subset \{2,3,4\}  \\
         \{4\}\subset \{1,2,4\}\quad &\quad \{4\}\subset \{1,3,4\}\quad &\quad \{4\}\subset \{2,3,4\}.  \\
    \end{array}
\end{equation*}
\end{ex}

\subsection{Positroids}\label{sec:positroid}

Let $\mathbb F$ be a field and let $A$ be a $k\times n$ matrix with entries in $\mathbb F$. Let $I\subset[n]$ such that $|I|=k$. We think of the set $[n]$ as indexing the columns of $A$ and thus the set $I$ is a $k$-subset of the columns.  Let $\Delta_I(A)$ denote the determinant of the $k\times k$ submatrix of $A$ given by the columns in $A$ indexed by $I$.

Let $M=([n],\mathcal B)$ be a matroid of rank $r_M=k$. We say that $M$ is \emph{representable over $\mathbb F$} if there exists a full rank $k\times n$ matrix $A$ with entries in $\mathbb F$ such that $B\in\mathcal B(M)$ if and only if $\Delta_B(A)\neq 0$. In this way, we say that the matrix $A$ \emph{represents} the matroid $M$ over $\mathbb F$. We denote the matrix $A$ by $A_M$ to highlight the representation property. In this context the determinants $\Delta_B(A_M)$ are commonly called the Pl\"ucker coordinates of $M$. On the other hand, given a full rank $k\times n$ matrix $A$ with entries in $\mathbb F$, we construct the matroid $M_A$ by its set of bases $\mathcal B (M_A)$, where $\mathcal B_A=\left\{B\in{\mleft(\genfrac..{0pt}{}{[n]}{k}\mright)}\,:\, \Delta_B(A)\neq 0\right\}$. We say that a matroid $M$ is \emph{representable} if $M$ is representable over $\mathbb R$. Note that the matrix $A_M$ may not be unique since any elementary operation on rows does not affect the matroid structure nor the Pl\"ucker coordinates.

\begin{defi}\label{positroid} A matroid $P=([n],\mathcal{B})$ is called a $\emph{positroid}$ if $P$ is representable via a matrix $A_P$ whose maximal minors are nonnegative.
\end{defi}

Positroids are of particular interest as they have a strong connection to the positive Grassmannian. The \emph{Grassmannian} $Gr_{k,n}(\mathbb R)$ is the set of $k$-dimensional vector subspaces $V$ in $\mathbb R^n$. Such a subspace $V$ can be thought of as a $k\times n$ full dimensional matrix $A$ by taking a basis of $V$ as the rows of $A$. In this way, we can think of $Gr_{k\times n}(\mathbb R)$ as full dimensional $k\times n$ matrices over $\mathbb R$ modulo left multiplication by full dimensional $k\times k$ matrices. Now, given a matroid $M=([n],\mathcal B)$ let $S_M=\{A\in Gr_{k\times n}\,:\, \Delta_I\neq 0 \text{ iff }I\in\mathcal B(M), \text{ for all maximal minors }\Delta_I\}$. Then $$
Gr_{k\times n}(\mathbb R) = \bigsqcup_{M}S_M
$$
where the union is over all matroids $M$ on $[n]$ of rank $k$. This decomposition, however, is not a stratification (\cite{juggling},\cite{GelfandSerganova}). On the other hand, if we restrict ourselves to $Gr_{k\times n}^{\geq 0}$ which is the \emph{nonnegative part of $Gr_{k\times n}(\mathbb R)$}, that is, $$Gr_{k\times n}^{\geq 0}=\{A\in Gr_{k\times n}\,:\Delta_I\geq 0 \text{ for all maximal minors }\Delta_I\}$$ we get a nicer structure. Postnikov \cite{pos}, \cite{pos2} proved that carrying the decomposition of the Grassmannian in matroids to the nonnegative Grassmannian with positroids, provides a stratification for $Gr_{k\times n}^{\geq 0}$. That is,
$$
Gr_{k\times n}^{\geq 0} = \bigsqcup_{P}S^{\geq 0}_P
$$
 where the union is over all positroids $P$ on $[n]$ of rank $k$ and $S_{P}^{\geq 0}$ consists of those $A\in Gr_{k\times n}^{\geq 0}$ such that $\Delta_I(A)>0$ if and only if  $I\in\mathcal B(P)$.

\begin{ex}\label{ex:Positroid}
Let $$A = \left(\begin{matrix} 1 & 1 & 1 & 1 & 1 \\
1 & 2 & 3 & 4 & 5\\
1 & 4 & 9 & 16 & 25
\end{matrix}\right).$$

The reader can check that $M_A$ is a positroid as each of the 10 maximal minors of $A$ is nonnegative. In fact, $M_A$ coincides with the matroid $U_{3,5}$. This is not accident, as the following lemma shows that every uniform matroid is a positroid.

\end{ex}

\begin{lem}\label{lem:uniform are positroids}
    Let $k,n$ be integers such that $k\leq n$. The uniform matroid $U_{k,n}$ is a positroid.
\end{lem}

\begin{proof}
    Take $a_1,\ldots,a_n\in\R$ such that $0<a_1<\ldots<a_n$ and consider the $k\times n$ matrix \begin{equation*}
        A=\left[\begin{array}{cccc}
             1 & 1 & \cdots & 1  \\
             a_1 & a_2 & \cdots & a_n \\
             a_1^2 & a_2^2 & \cdots & a_n^2 \\
             \vdots & \vdots & \ddots & \vdots \\
             a_1^{k-1} & a_2^{k-1} & \cdots & a_n^{k-1} \\
        \end{array}\right].
    \end{equation*} Any maximal minor of $A$ is a Vandermonde matrix so for all $I\in{\mleft(\genfrac..{0pt}{}{[n]}{k}\mright)}$, $$\Delta_I(A)=\prod_{\substack{i_1<i_2 \\ i_1,i_2\in I}}(a_{i_2}-a_{i_1}).$$ As $i_1<i_2$ implies that $a_{i_1}<a_{i_2}$, we have that $a_{i_2}-a_{i_1}>0$ and $\Delta_I(A)$ is a product of positive numbers. Therefore all Pl\"ucker coordinates of $A$ are positive and all collections of at most $k$ columns of $A$ form linearly independent sets. This gives a direct bijection between $M(A)$ and $U_{k,n}$ that preserves independent sets. That is, a matroid isomorphism. Since $A$ has positive maximal minors and is isomorphic to $U_{k,n}$ we conclude that $U_{k,n}$ is a positroid.
    
\end{proof}

Notice that, unlike arbitrary representable matroids, positroids depend heavily on an ordering of the ground set, as changing the order of the columns of a matrix $A$ can change the sign of its minors. 

\subsection{Grassmann necklaces and decorated permutations}

The combinatorial objects presented here appear in \cite{pos} as part of the family of combinatorial objects parametrizing positroids. In order to define these objects we make use of the \emph{$i$-order $<_i$ on $[n]$} which is the total order given by \begin{equation*}
    i <_i i+1 <_i \cdots <_i n <_i 1 <_i 2 <_i \cdots <_i i-1.
\end{equation*} 
Now let $S=\{s_1<_i\cdots<_i s_k\}$ and $T=\{t_1<_i\cdots<_i t_k\}$ be $k$-subsets of $[n]$ totally ordered using $<_i$. We say that  $S\leq_i T$ in the \emph{$i$-Gale order} if and only if $s_j\leq_i t_j$ for all $j\in[k]$. With this in hand, we can define the first combinatorial object we will need and describe how it indexes positroids. 

\begin{defi} Let $0\leq k\leq n$. A \emph{Grassmann necklace} of type $(k,n)$ is a sequence $I=(I_1,\ldots,I_n)$ of subsets $I_i\in  {\mleft(\genfrac..{0pt}{}{[n]}{k}\mright)} $ such that for any $i\in [n]$ \begin{itemize}
    \item if $i\in I_i$, then $I_{i+1}=(I_i\setminus\{i\})\cup\{j\}$ for some $j\in[n]$,
    \item if $i\notin I_i$, then $I_{i+1}=I_i$,
\end{itemize} and $I_{n+1}:=I_1$.
\end{defi}

As stated before, a positroids can be indexed by means of \emph{Grassmann necklaces}. This means that Grassmann necklaces capture enough information through their sequence of sets to completely recover all positroids. Namely, the Grassmann necklace corresponding to $P$ consists of some of the basis of $P$, but these bases are enough to obtain them all, using $i$-Gale orders as Propositions \ref{prop:PostoGN} and \ref{prop:GNtoPost} describe.

\begin{prop}\label{prop:PostoGN}\cite[Lemma 16.3]{pos} Given a matroid $P=([n],\mathcal{B})$ of rank $k$, let $I_i$ be the minimal element of $\mathcal{B}$ with respect to the $i$-Gale order. Then $\mathcal{I}(P)=(I_1,\ldots,I_n)$ is a Grassmann necklace of type $(k,n)$.
\end{prop}

\begin{rem}\label{rem:GNandMatroids}
Consider the matrices $${A}=\left(\begin{matrix}1&1&0&-2\\0&1&1&2\end{matrix}\right)\qquad {B}=\left(\begin{matrix}1&0&1&0\\0&1&2&1\end{matrix}\right).$$ The reader can check that the matroids $M_{A}$ and $M_B$ have bases $\mathcal{B}(M_A)=\{12,13,14,23,24,34\}$ and $\mathcal{B}(M_B)=\{12,13,14,23,34\}$ respectively and that their corresponding Grassmann necklaces are $\mathcal{I}(M_A)=(12,23,34,41)$ and $\mathcal{I}(M_B)=(12,23,34,41)$. Although the matroids $M_A$ and $M_B$ have the same Grassmann necklace, $M_A$ is a positroid but $M_B$ is not. That is, although $M_B$ is representable there is no representation of $M_B$ inside $Gr_{k,n}^{\geq  0}$.    
\end{rem}

When $M$ is a positroid then its Grassmann necklace allows us to recover all of its bases in the following way.

\begin{prop}\label{prop:GNtoPost}\cite{oh,pos}  Given a Grassmann necklace $I=(I_1,\ldots,I_n)$ of type $(k,n)$, the set $$\mathcal{B}(I)=\left\{B\in  {\mleft(\genfrac..{0pt}{}{[n]}{k}\mright)}  \,:\, B \geq_i I_i \text{ for all } i\in [n]\right\}$$ is the collection of bases of a positroid $\mathcal{P}(I)=([n],\mathcal{B}(I))$ of rank $k$. 
\end{prop}

\begin{rem}
One can check from Propositions \ref{prop:PostoGN} and \ref{prop:GNtoPost} that if $M$ is a positroid and $I(M)$ its Grassmann necklace then $M=\mathcal{P}(\mathcal{I}(M))$ and $I=\mathcal{I}(\mathcal{P}(I))$. Moreover, if $M$ is not a positroid then $M\subsetneq \mathcal P(\mathcal{I}(M))$ and therefore $P(\mathcal{I}(M))$ is the smallest positroid that contains all matroids with the same Grassmann necklace.  This is illustrated with the matroid $M_B$ in Remark \ref{rem:GNandMatroids}.
\end{rem}

Making use of Grassmann necklaces we define \emph{decorated permutations}. These will be used as a second gadget to index positroids. 


\begin{defi} A \emph{decorated permutation} is a bijection $\sigma:[n]\rightarrow[n]$ whose fixed points are decorated as $\sigma(\ell)=\underline{\ell}$ or $\sigma(\ell)=\overline{\ell}$.  We denote the set of all decorated permutations on $[n]$ by $\mathcal D_n$.
\end{defi}

Given $\sigma\in \mathcal D_n$ and $i\in[n]$ we say that $j\in[n]$ is a \emph{weak} $i$-\emph{excedance} of $\sigma$ if $\sigma(j)=\overline{j}$ or $j<_i \sigma(j)$.
We denote by $W_i(\sigma)$ the set of weak $i$-excedances of $\sigma$. It can be shown that $|W_i(\sigma)|=|W_j(\sigma)|$ for all $i,j\in[n]$ \cite{williams}. We thus denote the number of weak excedances of $\sigma$ by $W(\sigma):=|W_1(\sigma)|$. For instance, if $\sigma=5\underline{2}6134$ then $W_1(\sigma)=\{1,3\}$ and thus $W(\sigma)=2$.

Let $\mathcal D_{k,n}$ be the set of all decorated permutations on $[n]$ with $k$ weak excedances. Then $\mathcal D_n=\sqcup_{k=0}^n\mathcal D_{k,n}$. The cardinality of $D_{k,n}$ has been computed in \cite[Proposition 23.1]{pos}  and corresponds to the sequence \href{https://oeis.org/A046802}{A046802} \cite{OEIS}. The following proposition gives us a bijective way to go between Grassmann necklaces and decorated permutations. If $\sigma$ is a decorated permutation we will denote by $P_\sigma$ its corresponding positroid.

\begin{prop}\label{prop:dpTOgn}
\cite[Proposition 4.6]{ARW}\label{arw:prop} Let $I=(I_1,\ldots,I_n)$ be a $(k,n)$ Grassmann necklace. Let $\pi(I)$ be the decorated permutation in $\mathcal D_{k,n}$ constructed as follows: \begin{itemize}
    \item if $I_{i+1}=(I_i\setminus\{i\})\cup\{j\}$ with $i\neq j$, then $\pi(I)(j):=i$.
    
    \item if $I_{i+1}=I_i$ where $i\in I_i$, then $\pi(I)(i)=\overline{i}$.
    
    \item if $I_{i+1}=I_i$ where $i\notin I_i$, then $\pi(I)(i)=\underline{i}$.
\end{itemize} Conversely, let $\sigma$ be a decorated permutation on $[n]$ with $k$ weak excedances, and let $I_{i}$ be the set of weak $i$-excedances of $\sigma$. Then $\mathcal{I}(\sigma)=(I_{1},\ldots,I_{n})$ is a Grassmann necklace of type $(k,n)$.
In fact, these constructions are inverses of each other \cite{ARW}. Namely, $\pi(\mathcal{I}(\sigma))=\sigma$, and $\mathcal{I}(\pi(I))=I$.
\end{prop}

\begin{rem}
The bijection between Grassmann necklaces and decorated permutations outlined here, and used throughout the paper, is due to Postnikov in~\cite{pos}. Our results can also be expressed using a different bijection given in~\cite{oh}.
\end{rem}

\begin{ex} 
Consider the Grassmann necklace $I=(13,34,34,45,56,61)$. Using Proposition \ref{prop:dpTOgn} we get the decorated permutation $\pi(I)=5\underline{2}6134$. 
On the other hand, taking the Grassmann necklace $(12,23,34,41)$ that indexes the uniform positroid $M=U_{2,4}$ we get that its corresponding decorated permutation is $\pi_M=3412$.
\end{ex}

As part of our notation, we denote by $[a,b]$ the cyclic intervals of $[n]$. That is, the sets of form $\{a,a+1,\dots,b-1,b\}$ if $a\leq b$ and $\{a,a+1,\ldots,n,1,\ldots,b\}$ if $b<a$. This allows us to describe some details in a more compact way. 

The positroid $U_{k,n}$ is such that its Grassmann necklace is given by 
$$\mathcal{I}=([k], [2,k+1],\dots,[n,k-1])$$ and its decorated permutation is $$\pi_{k,n} := (n-k+1)(n-k+2)\cdots n12\cdots(n-k).$$ That is, $\pi(i)=n-k+i \,(mod \,n)$ for $i\in[n]$.

It follows from Proposition \ref{prop:dpTOgn} that $\sigma(i)=\underline i$ implies $\{i\}$ is a \emph{loop} (or 1-element circuit) of the associated positroid $P_\sigma$, and thus is never contained in an element of $\mathcal{I}(P_\sigma)$. Similarly, $\sigma(i)=\overline i$ implies $\{i\}$ is a \emph{coloop} of $P_\sigma$ and is in every element of $\mathcal{I}(P_\sigma)$ and thus in every basis of $P_\sigma$.

As stated in the introduction, we are interested in a combinatorial characterization of quotients of positroids. Thus our main question is, given two positroids $P_1$ and $P_2$ on the ground set $[n]$, can we determine combinatorially whether $P_1$ and $P_2$ form a quotient? We will make use of decorated permutations to give a partial answer to this question.


\section{Poset of positroid quotients}\label{poset:sec}
For every $n\geq 1$, we let $\mathcal P_n$ be the poset whose elements consist of the decorated permutations in $D_n$ and whose order relation is the transitive closure of the following  covering relation: $\tau\lessdot\pi$ if and only if $\tau\in D_{k-1,n}$, $\pi\in D_{k,n}$ for some $k\in [n]$, and $P_\tau$ is a quotient of $P_\pi$. We call $\mathcal P_n$ \emph{the poset of positroid quotients on $[n]$}. See Figure \ref{fig:Poset_of_Concordance} for an illustration of $\mathcal P_3$.

\begin{figure}[h]
\centering

\tikzset{every picture/.style={line width=0.75pt}} 

\begin{tikzpicture}[x=0.75pt,y=0.75pt,yscale=-1,xscale=1]

\draw  [line width=0.2mm]   (305.67,70) -- (155.67,110) ;

\draw  [line width=0.2mm ]   (305.67,70) -- (255,110.67) ;

\draw  [line width=0.2mm ]  (305.67,70) -- (305.67,109.33) ;

\draw  [line width=0.2mm ]  (354.91,110.93) -- (305.67,70) ;

\draw  [line width=0.2mm ]  (305.67,70) -- (405.67,110.67) ;

\draw  [line width=0.2mm ]  (305.67,70) -- (455,110) ;

\draw  [line width=0.2mm ]  (305.67,70) -- (205.67,110) ;

\draw  [line width=0.2mm ]  (305.61,232) -- (456.33,190.67) ;

\draw  [line width=0.2mm ]  (305.61,232) -- (356.33,190.67) ;

\draw  [line width=0.2mm ]  (305.61,232) -- (305.67,189.33) ;

\draw  [line width=0.2mm ]  (305.61,232) -- (255.67,190.67) ;

\draw  [line width=0.2mm ]  (305.61,232) -- (206.12,190.11) ;

\draw  [line width=0.2mm ]  (305.61,232) -- (156.78,190.18) ;

\draw  [line width=0.2mm ]  (305.61,232) -- (406.33,189.33) ;

\draw  [line width=0.2mm ]  (305.67,130.33) -- (455,171) ;

\draw  [line width=0.2mm ]  (305.67,130.33) -- (355.67,170.33) ;

\draw  [line width=0.2mm ]  (305.67,130.33) -- (154.75,170.25) ;

\draw  [line width=0.2mm ]  (255.25,130.25) -- (405,170.33) ;

\draw  [line width=0.2mm ]  (255.25,130.25) -- (255.75,169.75) ;

\draw  [line width=0.2mm ]  (255.25,130.25) -- (154.75,170.25) ;

\draw  [line width=0.2mm ]  (204.75,129.75) -- (255.75,169.75) ;

\draw  [line width=0.2mm ]  (204.75,129.75) -- (205.25,170.25) ;

\draw  [line width=0.2mm ]  (204.75,129.75) -- (154.75,170.25) ;

\draw  [line width=0.2mm ]  (154.25,130.75) -- (305.75,169.75) ;

\draw  [line width=0.2mm ]  (154.25,130.75) -- (205.25,170.25) ;

\draw  [line width=0.2mm ]  (154.25,130.75) -- (154.75,170.25) ;

\draw  [line width=0.2mm ]  (405,170.33) -- (406.33,129.67) ;

\draw  [line width=0.2mm ]  (406.33,129.67) -- (355.67,170.33) ;

\draw  [line width=0.2mm ]  (406.33,129.67) -- (305.75,169.75) ;

\draw  [line width=0.2mm ]  (204.75,129.75) -- (355.67,170.33) ;

\draw  [line width=0.2mm ]  (356.33,130.33) -- (455,171) ;

\draw  [line width=0.2mm ]  (356.33,130.33) -- (305.75,169.75) ;

\draw  [line width=0.2mm ]  (356.33,130.33) -- (255.75,169.75) ;

\draw  [line width=0.2mm ]  (456,130.33) -- (205.25,170.25) ;

\draw  [line width=0.2mm ]  (456,130.33) -- (405,170.33) ;

\draw  [line width=0.2mm ]  (456,130.33) -- (455,171) ;

\draw (305,245) node  [align=left] {$\underline{123}$};
\draw (155.67,180) node  [align=left] {$312$};
\draw (204.67,180) node  [align=left] {$21\underline{3}$};
\draw (255,180) node  [align=left] {$3\underline{2}1$};
\draw (306,180) node  [align=left] {$\underline{12}\overline{3}$};
\draw (355.67,180) node  [align=left] {$\underline{1}32$};
\draw (405.67,180) node  [align=left] {$\underline{1}\overline{2}\underline{3}$};
\draw (456,180) node  [align=left] {$\overline{1}\underline{23}$};
\draw (154.67,120) node  [align=left] {$21\overline{3}$};
\draw (205,120) node  [align=left] {$231$};
\draw (255.33,120) node  [align=left] {$3\overline{2}1$};
\draw (306,120) node  [align=left] {$\overline{1}32$};
\draw (356,120) node  [align=left] {$\overline{1}\underline{2}\overline{3}$};
\draw (405.67,120) node  [align=left] {$\underline{1}\overline{23}$};
\draw (456.33,120) node  [align=left] {$\overline{12}\underline{3}$};
\draw (305.33,60) node  [align=left] {$\overline{123}$};

\end{tikzpicture}
\caption{The poset of positroid quotients $\mathcal P_{3}$.}
\label{fig:Poset_of_Concordance}
\end{figure}
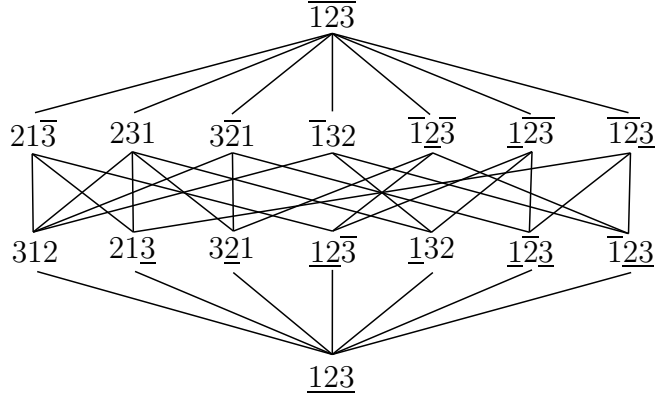

Let $\pi_{k,n}$ denote the decorated permutation corresponding to the uniform positroid $U_{k,n}$. That is, $\pi_{k,n}:=(n-k+1)\cdots (n-1)\:n\:1\:2\:\cdots (n-k)$. We state the following properties of the poset $\mathcal{P}_n$, whose proof we leave to the reader:
	
	\begin{enumerate}
		\item It is a poset with $\hat{0}$ given by the decorated permutation $\pi_{0,n}$ and $\hat{1}$ given by the decorated permutation $\pi_{n,n}$.
		
		\item It is graded and the rank of each decorated permutation is its number of weak excedances. The number of elements in each rank is recorded in the sequence \href{https://oeis.org/A046802}{A046802}. Thus, its rank polynomial is symmetric and unimodal.
		
		\item If $\sigma,\tau\in D_n$ are decorated permutations such that $\tau\leq\sigma$ and $\sigma(i)=\underline{i}$, then $\tau(i)=\underline{i}$. On the other hand, if $\tau(i)=\overline{i}$, then $\sigma(i)=\overline{i}$.
	\end{enumerate}

Similarly, one can construct the poset of matroid quotients $\mathcal M_n$  whose elements are all matroids on the ground set $[n]$ and whose order relation is $M<N$ if and only if $M\not=N$ and $M$ is a quotient of $N$. This is implicitly done in \cite{white}. Moreover, it is shown in \cite[Prop. 8.2.5]{white}, in the language of strong maps, that if $\{M,N\}$ form a flag matroid and $r_M<r_N$ then there exist matroids $M=M_0,M_1,\dots,M_k=N$ such that $r_{M_{i+1}}=r_{M_i}+1$ for $i=0,\dots,k-1$, and $\{M_0,\dots,M_k\}$ form a flag matroid. That is, any flag matroid $M\subset N$ can be extended to a saturated flag matroid $M=M_0\subset M_1\subset \cdots \subset M_k=N$ whose constituents $M_i$ have each possible rank between $r_M$ and $r_N$.

In other words, if $M< N$ in $\mathcal M_n$ then there is a saturated chain $M\lessdot M_1\lessdot\cdots\lessdot N$ in $\mathcal M_n$. The existence of such saturated chains is made explicit via the Higgs lift  (see \cite[Prop. 8.2.5]{white} for details). One may feel tempted to conclude the same in the poset $\mathcal{P}_n$. However, this is unclear as one needs to guarantee that the Higgs lift of a positroid, is again a positroid. Nonetheless, since $\mathcal P_n$ is a ranked poset, then if $\sigma,\tau\in\mathcal D_n$ are such that $\sigma\leq\tau$ then there is a saturated chain in $\mathcal D_n$ from $\sigma$ to $\tau$, although the existence of such saturated chains is not explicit. Thus we pose the following problem.

\begin{prob}
Let $\sigma\leq \tau$ in $\mathcal D_n$. How can we construct positroids $P_\sigma=P_{\sigma_1},\dots,P_{\sigma_{m-1}}=P_tau$ such that  $\sigma=\sigma_0\lessdot\sigma_1\lessdot\cdots\lessdot\sigma_m=\tau$ is a saturated chain in $\mathcal D_n$ where $m=r_{\tau}-r_{\sigma}$?.
\end{prob}

Our main theorem identifies a set of positroids that are quotients of $U_{k,n}$ for any $k\in[n]$. We do this by defining the following sequence of moves on decorated permutations. 

\begin{defi}
Given a decorated permutation $\pi\in D_n$ and a subset $A$ of $[n]$ we denote by $\overleftarrow{\rho_A}(\pi)$ the element of $D_n$ obtained from $\pi$ by performing the following moves in order:
\begin{itemize}
    \item[(F)]\textbf{Freeze move:} Freeze the value $i$ in $\pi$ if and only if $i\in A$.
    \item[(S)]\textbf{Shift move:} After freezing, cyclically shift the remaining values in $\pi$ one place to the left, jumping over frozen elements when necessary.
    \item[(D)]\textbf{Decoration move:} If any fixed point $i$ appears after the shift move (S), decorate it as $\underline i$.
\end{itemize}
Analogously, we denote by $\overrightarrow{\rho_A}(\pi)$ the permutation obtained from $\pi$ by performing the (F), (S), (D) moves, with the difference that (S) is shifts to the right instead of the left, and the (D) decorates any new fixed point i as $\overline{i}$. We call this sequence of moves a \emph{cyclic shift} of $\pi$.
\end{defi}

\begin{ex}
Let $\pi=34512$. Then
\begin{align*}
    \overleftarrow{\rho_{\emptyset}}(\pi)=45123\quad\quad & \quad\quad \overleftarrow{\rho_{12}}(\pi)=45\underline{3}12\\
    \overleftarrow{\rho_{25}}(\pi)=41532\quad\quad & \quad\quad \overleftarrow{\rho_{4}}(\pi)=54123\\
   \overrightarrow{ \rho}_{12}(\pi)=53412\quad\quad & \quad\quad  \overrightarrow{ \rho_{4}}(\pi)=24\overline{3}51\\
\end{align*}
\end{ex}

\begin{prop}
\label{prop:inverse and dual}
    Let $A\subseteq [n]$. Then $\overrightarrow{\rho_A}(\pi_{k,n})^{-1}=\overleftarrow{\rho_B}(\pi_{n-k,n})$, where $B$ is the set $\pi_{k,n}^{-1}(A)$.
    

\end{prop}

\begin{proof}

Denoting $\pi_{k,n}$ as $\pi$, notice that $\pi^{-1} = \pi_{n-k,n}$. Thus freezing $\pi(i)$ for every $\pi(i)\in A\subset [n]$ implies every element of $B = \{b \in [n]: b = \pi^{-1}(i) \text { for } i\in A\}$ is frozen in $\pi_{n-k,n}$.


Consider the remaining values that are not frozen. They form a permutation, $\omega$, over the set $[n]\setminus A$. The permutation $\omega$ is indexed by the ordered set $[n]\setminus B$. Then a cyclic shift to the right of the values of $\omega$ is equivalent to a cyclic shift to the left of the indices of $\omega$. But for $i\not\in B$, the index $\omega^{-1}(i)$ is precisely the entry $\pi^{-1}(i)=\pi_{n-k,n}(i)$. Thus, $\overrightarrow{\rho_{A}}(\pi)^{-1} = \overleftarrow{\rho_B}(\pi^{-1})$.
\end{proof}

    
    


\subsection{Uniform quotients}

Now we will characterize most positroids of rank $k-1$ that are quotients of the uniform positroid $U_{k,n}$ through its decorated permutation $\pi_{k.n}$. For the remainder of the paper we will consider $A\subseteq[n]$ as a union of disjoint cyclic intervals of $[n]$ so that $A = [a_1,i_1]\cup \cdots \cup [a_m,i_m]$, denoting singletons as $\{a_j\}$. 
We call each of the cyclic intervals a \emph{cyclic component} of $A$. For example if $n=13$ and $A=\{1,2,5,8,9,12,13\},$ then $A=[12,2]\cup \{5\}\cup [8,9]$.



\begin{thm}\label{thm:rank}
Let $k\leq n$ and $A\subset[n]$ be a union of disjoint cyclic intervals of $[n]$ such that no interval in $A$ has size greater than $k-1$. Then the positroid represented by $\overleftarrow{\rho_A}( \pi_{k,n})$ has rank $k-1$.
\end{thm}

\begin{proof} 

Let $\pi=\pi_{k,n}$, $\sigma:=\overleftarrow{\rho_A}(\pi)$, and $P_\sigma$ be the positroid represented by $\sigma$. Since the number of weak excedances of a decorated permutation is equal to the rank of its corresponding positroid, it is enough to show that $|W_1(\sigma)|=k-1$. Recall that $W_1(\pi)=[k]$ and $\pi(x)=n-k+x$ for $x\in[1,k]$ and $\pi(x)=x-k$ for $x\in[k+1,n]$. We will show that $W_1(\sigma)=W_1(\pi)\setminus \{j\} $, where $j=\max\{i\in[k]\,:\, \pi(i)\notin A\}$. 

Let $l\in[k]$. If $\pi(l)\in A$, then $\sigma(l)=\pi(l)$ and $l\in W_1(\sigma)$. On the other hand, if $\pi(l)\notin A$, and $l\neq j$, then there exists $l'=min\{i\in[k]\,:\,l< i \text{ and } \pi(i)\notin A\}$. Since $\pi(l)=n-k+l$ and $\pi(l')=n-k+l'$ we have that $\sigma(l)=n-k+l'$. As $l'\in [k]$, we have that $l<l'<n-k+l'$ and therefore $l\in W_1(\sigma)$. This holds in particular for $l'=j$. Hence $W_1(\pi)\setminus \{j\}\subseteq W_1(\sigma)$.

Let $j'=\min\{i\in[n]\setminus [k]\,:\, \pi(i)\notin A\}$ if it exists. Since $\pi(j)=n-k+j$ and $\pi(j')=j'-k$, we have that $\sigma(j)=j'-k$. Recall intervals of $A$ have length less than $k$. Then $j'-j\leq k$, implying that $j'-k\leq j$. Since $\sigma(i)=\underline{i}$ for any fixed $i$, it follows $j\not\in W_1(\sigma)$. 


So far we have shown that $W_1(\pi)\setminus \{j\} \subseteq W_1(\sigma)$ and $j\not\in W_1(\sigma)$. We complete the proof by showing $[k+1,n]\not\subset W_1(\sigma)$.

Let $l\in[k+1,n]$ and $\sigma(l)=l'$ for some $l'$. If $l<_1 l'$ then $\pi(l')=l'-k$ and $\sigma(l)=l'-k$. Since $l'-l\leq k$, then $l'-k\leq l$ and thus $l\notin W_1(\sigma)$. 

On the other hand, if $l'\leq k$ then $\sigma(l)=n-k+l'$. Again, the cyclic interval $[l',l]$ contains at most $k-1$ elements. Thus $n+l'-l\leq k$ and $n-k+l'\leq l$ and we conclude that $l\notin W_1(\sigma)$. Therefore $W_1(\sigma)=W_1(\pi)\setminus j$ and $P_\sigma$ has rank $k-1$, as desired.

\end{proof}

\begin{coro}\label{prop:rank}
Let $k\leq n$ and let $A\subset[n]$ such that $|A|\leq k-1$. Then the positroid represented by the decorated permutation $\overleftarrow{\rho_A}(\pi_{k,n})$ has rank $k-1$.
\end{coro}


We are now ready to state our main result.

\begin{thm}[Main theorem]\label{main:theo}
 Let $\sigma\in D_{k-1,n}$ be a decorated permutation of rank $k-1$ and choose $A\in {\mleft(\genfrac..{0pt}{}{[n]}{\ell}\mright)}$, for some $\ell\in\{0,\dots,k-1\}$. If $\sigma=\overleftarrow{\rho_A}(\pi_{k,n})$, then $\sigma\lessdot\pi_{k,n}$. 
\end{thm}

The proof of Theorem \ref{main:theo} relies on showing that every circuit of the uniform positroid represented by the decorated permutation $\pi_{k,n}$ is a union of circuits of the positroid represented by the decorated permutation $\overleftarrow{\rho_A}(\pi_{k,n})$ for the set $A\subset [n]$. To this end, we state the following propositions.

We first show that the set of circuits of the positroid $P_{\sigma}$ represented by $\sigma = \overleftarrow{\rho_A}(\pi_{k,n})$ has a simple description. Then we show how to obtain the circuits of $\pi_{k,n}$ as the union of circuits of $P_{\sigma}.$ 

\begin{thm}[Circuit description of shifted uniform positroid]\label{ShiftedUniformCircuits}
Let $A=[a_1,i_1]\cup\cdots\cup[a_m,i_m]$ be a subset of $[n]$ composed of disjoint cyclic intervals of lengths $l_1,\ldots,l_m$, respectively. Then the circuits of the positroid $P_{\sigma}$ represented by $\sigma = \overleftarrow{\rho_A}(\pi_{k,n})$ are given by the set 
\begin{equation}\label{circuits_sigma}
    \mathcal C_A =\left\{ [i_{j}\!+\!1,i_{j}\!+\!k\!-\!l_j]\,:\,j=1,\dots ,m\right\}
    \cup
    \left\{C\in {\mleft(\genfrac..{0pt}{}{[n]}{k}\mright)} \,:\,[i_{j}+1,i_{j}+k-l_j]\not\subset C\right\}.
\end{equation}
Moreover, the circuits of size less than $k$ can be read from the decorated permutation $\sigma$. They are precisely the intervals $[i_{j}+1,\sigma(i_{j}+1)].$
\end{thm}

\begin{proof}
Let $\pi=\pi_{k,n}$. Assume that $\sigma=\overleftarrow{\rho_A}(\pi_{k,n})$ for some $A\in {\mleft(\genfrac..{0pt}{}{[n]}{\ell}\mright)} $ and $\ell=0,\dots,k-1$. 
We begin by proving that each interval $[i_j+1,i_j+k-l_j]$ satisfies $\sigma(i_j+k-l_j)=i_j+1$ and that for all $r\in[n]$, the interval $[\sigma(j_r),j_r]$ is a circuit of $P_{\sigma}$. Let us suppose that among the frozen set $A$, one of its intervals is $[a,b]$ of length $b-a+1$. Notice that the interval $[b+1,a+k-1]$ is the cyclic interval that would extend (clockwise) $[a,b]$ into a bigger interval of length precisely $k$. In fact $[b+1,a+k-1]$ is the description of $[i_j+1,i_j+k-l_j]$ in terms of $a$ and $b$. As we are freezing the interval $[a,b]$ and $\pi(x)=x-k$ (mod $n$), we have that we are freezing the positions $[a+k,b+k]$ (mod $n$). Therefore we have that if
\begin{equation*}
    \pi=\left(\begin{array}{ccccccc}
        \cdots & a+k-1 & a+k & \dots & b+k & b+k+1 & \cdots \\
        \cdots & a-1 & a & \dots & b & b+1 & \cdots
    \end{array}\right),
\end{equation*} and \begin{equation*}
    \sigma=\left(\begin{array}{ccccccc}
        \cdots & a+k-1 & a+k & \dots & b+k & b+k+1 & \cdots \\
        \cdots & b+1 & a & \dots & b & * & \cdots
    \end{array}\right).
\end{equation*} and $\sigma(a+k-1)=b+1$ and $\sigma(i_j+k-l_j)=i_j+1$. For the rest of the proof we will denote such intervals as $[\sigma(j_r),j_r]$ for some $r\in[n]$.


We will now show the circuits described above and the k-subsets not containing these intervals are the only circuits of $P_{\sigma}$.

If $A=\emptyset$ then $\mathcal{C} = {\mleft(\genfrac..{0pt}{}{[n]}{k}\mright)}$, the circuits of the uniform positroid $U_{k-1,n}$ represented by $\sigma$.

Now let $\ell\in[k-1]$ and $A\in{\mleft(\genfrac..{0pt}{}{[n]}{\ell}\mright)}$. As before, we have that $W_1(\pi)=[k]$. Moreover, for each $r\in[n]$ it holds that $W_r(\sigma)=\{r,\ldots,j_r-1,j_r+1,\ldots,r\!+\!k\!-\!1\}=[r,r+k-1]\setminus\{j_r\}$ where $j_r=\max\{i\in[r,r\!+\!k\!-\!1]\,:\, \pi(i)\notin A\}$. That is, $j_r$ is the largest among the first $k$ positions in $r$-Gale order such that $\pi(j_r)$ is not frozen. Therefore, the Grassmann necklace $\mathcal{I}_\sigma=(I_1,\dots,I_n)$  of $\sigma$ is such that $I_i=W_i(\sigma)$. 
This means, in particular, that $I_i$ is the minimal basis of $P_\sigma$ with the $i$-Gale order, for each $i$. 
Let us see that for each $r\in[n]$ the interval $[\sigma(j_r),j_r]$ is a circuit of $P_{\sigma}$. We illustrate the proof with $r=1$ as it is done analogously for each $r$. Set $j=j_1$ and notice that since $j\notin W_1(\sigma)$ then $\sigma(j)<j$ and $[\sigma(j),j]\subseteq [1,k]$. If $[\sigma(j),j]$ were independent, then $[\sigma(j),j]$ would be contained in ${I}_{\sigma(j)}$ but $I_{\sigma(j)}$ avoids $j$ since $j_{\sigma(j)}=j$. Therefore $[\sigma(j),j]$  is dependent in $P_\sigma$ and now we show it is a circuit. To this end we will see that each of the sets $J_x=[\sigma(j),j]\setminus \{x\}$, for each $\sigma(j)\leq x\leq j$, is independent  by constructing a basis $B_x$ of $P_{\sigma}$ such that $J_x\subseteq B_x$. If $x=j$ then take $B_x=I_1$. If $x\neq j$ let us prove that $B_x=I_1\cup \{j\}\setminus \{x\}$ is a basis of $P_{\sigma}$.

Since
$B_x=[1,x-1]\cup[x+1,k]$ then $B_x\geq_r I_1$ for $r=1,\dots,x$ as $I_1=[1,j-1]\cup[j+1,k]$. Therefore, $B_x\geq_r I_r$ for $r=1,\dots,x$. It remains to show that $B_x\geq_r I_r$ for $r=x+1,\dots,n$. Recall that $I_r=[r,r+k-1]\setminus j_r$ for $r\in [n]$, where $j_r=\max\{i\in[r,r+k-1]\,:\,\pi(i)\notin A\}$. Arranging the elements of $I_r$ and $B_x$ using the $r$-Gale order we get that $I_r=\{r,\dots,j_r-1,j_r+1,\dots,r+k-1\}$ for any $r$, whereas for $x\leq r\leq k$ we get that $B_x=\{r,\dots,k,1,\dots,x-1,x+1,\dots,r-1\}$ and for $k+1\leq r\leq n$ we get $B_x=\{1,\dots,x-1,x+1,\dots,k\}$. In either case one can see that $B_x\geq_r I_r$. This allows us to conclude that $[\sigma(j_r),j_r]$ is a minimally dependent set in $P_{\sigma}$, for every $r\in[n]$, and therefore a circuit. The reader can verify that each of the sets 
$[i_{j}\!+\!1,i_{j}\!+\!k\!-\!l_j]$ from equation (\ref{circuits_sigma}) are of the form  $[\sigma(j_r),j_r]$ for some $r$. 

Now we show that any $k$-subset of $[n]$ that does not contain a $[\sigma(j_r),j_r]$ for some $r$, is a circuit.
Let $D$ be a $k$-subset of $[n]$ such that $D$ does not contain any of the sets $[\sigma(j_r),j_r]$ as given above. Given a $k-1$ subset $F$ of $D$ let us see that $F$ is a basis of $\sigma$. If $F$ were not a basis then there is $r\in[n]$ such that $F\not\geq_r I_r.$ Then, if  $F=\{c_1,\dots,c_{k-1}\}$ and $I_r=\{b_1<_r\cdots<_r b_{k-1}\}$ in $r$-Gale order, the assumption that $F\not\geq_r I_r$ implies that for some $l$, $c_p\geq b_p$ for $p\in[l-1]$ and $c_l<b_l$. However, $I_r=[r,j_r-1]\cup[j_r+1,r+k-1]$ so $l$ must be the position of $j_r$, which implies that $c_l=j_r$. Thus $[r,j_r]\subset F$ and then $[\sigma(j_r),j_r]\subset D$, which is a contradiction. 

So far we have shown that $\mathcal C_A$ is contained in $\mathcal C_{\sigma}$, the set of circuits of the positroid $P_{\sigma}.$ To prove the reverse containment we will see that if $S\notin \mathcal C_A$ then $S\notin\mathcal C_{\sigma}$. First, notice that we only need to consider $S\subset [n]$ such that $|S|<k$ and $S$ does not contain any interval $[\sigma(j_r),j_r]$. We will prove that such $S$ can be extended to a set $D$ of cardinality $k$ such that $D$ does not contain any of the $[\sigma(j_r),j_r]$ and therefore, $S$ will be independent. Suppose that the decomposition of the set $A$ in cyclic intervals is $A=J_1\cup\cdots \cup J_s$. Then, each of the $J_i$ gives rise to a cyclic interval $L_i$ such that $J_i\cup L_i$ is a cyclic interval of size $k$. Moreover, each of the $L_i$'s is a circuit of $\sigma$.

In order to prove that $S$ can be extended to the wanted $D$ we will make use of $J_1\cup L_1$ to this end. Let $D'=J_1\cup L_1\cup S$.

\emph{Case 1:} Suppose $D'$ contains no $L_j$ except $L_1$. If $|S\setminus(J_1\cup L_1)|\geq 1$ then $S$ can be extended to a $k$-subset $D$ of $D'\setminus\{a\}$ where $a\in L_i\setminus S$ does not contain any of the intervals in $\mathcal C_A$. On the other hand, if $S\subset L_1\cup J_1$ then take $D=J_1\cup L_1\setminus \{a\}\cup\{b\}$ where $a\in L_1\setminus S$ and $b:=c+1$ where $c$ is the clockwise greatest element of $L_1$. Notice that such $a$ exists as $S$ does not contain $L_1$ by hypothesis. If such a $b$ does not exist, it means that $D'=[n]$ and thus $\sigma\lessdot\pi=U_{n,n}$ as any matroid is concordant to $U_{n,n}$. Thus if $b$ does not exists there is nothing to prove. Otherwise, such $b$ exists and $D$ has size $k$, contains $S$ and none of the intervals in $\mathcal C_A$. Therefore $S$ is independent.

\emph{Case 2:} Suppose $D'$ contains $L_1$ and another different $L_j$. Without loss of generality call it $L_2$. If this is the case, then either $L_1\cup L_2$ or, $J_1\cup L_2$ is a cyclic interval. In the former case, then $|D'|\geq k+1$ and $D$ can be obtained by removing sequentially (until we get size $k$) from $D'$ elements $a\in (L_1\cup L_2) \setminus S$. These elements exist as $L_2\not\subset S$ and $L_1\cap L_2\neq\emptyset.$ In the latter case, $D$ can be obtained as a subset of $D''=(J_1\setminus\{d\})\cup (L_1\setminus \{a\})\cup\{b\}\cup S$ where $d\in (J_1\cap L_2)\setminus S$, $a\in (L_1\setminus S)$ and $d:=c+1$ where $c$ is the clockwise greatest element of $L_1$. As $|D|>k$ and does not contain any $L_j$, but does contain $S$, we can extend $S$ to a basis and make it an independent set.

\emph{Case 2.1:} If simultaneously, $D'$ contains $L_2$ and $L_3$ such that $L_1\cup L_2$ and $J_1\cup L_3$ are a cyclic intervals, then either $J_3\cap L_1=\emptyset$ or $J_3\cap L_1\neq\emptyset$. In the former case, remove form $D'$ any pair of elements $a\in (L_3\cap J_1)\setminus S$ and $b\in (L_2\cap L_1)\setminus S$ in which case $D'\setminus\{a,b\}$ will have still at least $k$ elements as the cyclic components of $S$ that intersect $L_2$ and $L_3$ have elements outside of $D'$. 
On the other hand, if $J_3\cap L_1\neq\emptyset$ then in order to maintain cardinality at least $k$ of $D'\setminus\{a,b\}$ with $a,b$ as above, we would need to guarantee that such elements can be substituted to keep the cardinality $k$ requirement. This can be achieved if $[n]\setminus J_1\cup L_1$ has at least 2 elements. If $|[n]\setminus J_1\cup L_1|=0$ we fall again in the $U_{n,n}$ case. If instead $|[n]\setminus J_1\cup L_1|=1$ then $J_3\cap J_1\neq\emptyset$ (as the reader can check) which is a contradiction as the $J_i's$ are disjoint. This exhausts all the possibilities and the proof is complete. 

\end{proof}



\begin{prop}\label{lemma:union_circuits}
Every circuit of $\pi_{k,n}$ can be obtained as a union of elements in $\mathcal C_A$.
\end{prop}

\begin{proof}
As $\pi_{k,n}$ corresponds to the uniform positroid $U_{k,n}$, all of its circuits are all the $k+1$-subsets of $[n]$. Let $O$ be any such circuit and let $$\mathcal C_O=\left\{C\in {\mleft(\genfrac..{0pt}{}{[n]}{k}\mright)} : C\in \mathcal{C}_A,  C\subset O \right\}.$$ If $|\mathcal C_O|\geq2$ then any two elements in $\mathcal C_O$ cover $O$. If $|\mathcal C_O|=1$, then $\mathcal C_O=\{C\}$ for some $C\in {\mleft(\genfrac..{0pt}{}{[n]}{k}\mright)}$ and there is one remaining element $x$ in $O$ we have yet to cover. Let us take the $k$-subset $D:=(C\setminus \{y\})\cup \{x\}$ for any $y\in C$. As $D\neq C$ and $|D|=k$, then $D\notin \mathcal C_O$. Thus there is a cyclic interval $L\in\mathcal C_A$ such that $x\in L$ and $L\subset D$. Thus, $O=L\cup C$ and the claim is proved in this case.

Finally, if $\mathcal C_O=\emptyset$ then every $k$-subset $C$ of $[n]$ contained in $O$ properly contains at least one of the cyclic intervals $[i_{s_j}+1,i_{s_j}+k-l_j]$ of $A$. Moreover, $\mathcal C_O$ contains at least 2 distinct intervals $L_1$ and $L_2$. To see this, take any $k$-subset $C_1=O\setminus \{x\}$ where $x\in O$ and let $L_1\subset C_1$ be an interval in $\mathcal C_A$. Now let $y\in L_1$ and set $C_2=O\setminus \{y\}$. Since $C_2\in\mathcal C_O$, there is an  interval $L_2\in\mathcal C_A$ such that $L_2\subseteq C_2$. Since $y\in L_1\subseteq C_1$ and $x\notin C_1$, then $L_1\neq L_2$. 
 
 Now assume $L_1,\dots,L_m$ are the intervals in $\mathcal C_A$ contained in $O$. Denote $L=L_1\cup \cdots\cup L_m$. Clearly $L\subseteq O$. We will prove the reverse containment. First notice that all the intervals in $L$ are pairwise disjoint. If this were not the case and if $|L|<|O|=k+1$ then without lose of generality we can assume $L_1\cap L_2\neq \emptyset$. As $L_1, L_2\in\mathcal C_A$ then there exist two disjoint cyclic intervals $J_1,J_2\in A$ that give rise to $L_1,L_2$. That is, for $r=1,2$, $J_r\cup L_r$ is a cyclic interval of length $k$. Assume that the least element in $L_1$ is less that the one in $L_2$. Then  $J_2\cup I_2 \subset L_1\cup L_2$ (otherwise $J_1$ overlaps $J_2$) which implies that $k<|L_1\cup L_2|\leq k+1$ and thus $O=L_1\cup L_2$. Therefore $O$ can not have overlapping intervals unless these intervals cover $O$.
 
 Now, suppose the intervals in $L$ are pairwise disjoint. Denote by $J_i\subseteq A$ the frozen interval that gives rise to $L_i$. We know that $\sum_{i=1}^m|J_i|\leq k-1$ and $|J_i|+|L_i|=k$ for all $i\in[m]$. Thus we get that $\sum_{i=1}^m|J_i|+\sum_{i=1}^m|L_i|=mk$ and $$\sum_{i=1}^m|L_i|=mk-\sum_{i=1}^m|J_i|\geq mk-(k-1)=(m-1)k+1\geq k+1$$ as $m\geq 2$. But since $L\subseteq O$ and $|O|=k+1$, we get that $L=O$. This finishes the proof.
 
\end{proof}

Theorem \ref{main:theo} follows immediately as a consequence of Theorem \ref{ShiftedUniformCircuits} and Proposition \ref{lemma:union_circuits}.

\begin{ex}[Illustrating Theorem \ref{main:theo}.]
Let $n=9$, $k=5$ and $A=[9,1]\cup(6)$. Then $\sigma=768291345$, the circuits $[\sigma(j_r),j_r]$ and the components $I_r$ of the Grassmann necklace $I_{\sigma}$ are given in the following table.

\begin{table}[h]
\centering
\begin{tabular}{|c|c|c|c|c|c|}
\hline
 $r$& 1 &2  &3  &4&5  \\
 \hline
$[\sigma(j_r),j_r]$ & [2,4] & [2,4] & [3,7] &[4,8] &[5,9] \\
\hline
$I_r$ & $[1,3]\cup\{5\}$ & $[2,3]\cup[5,6]$ & [3,6] &[4,7] &[5,8] \\
 \hline
 \hline
 $r$& 6 &7  &8  &9&  \\
 \hline
$[\sigma(j_r),j_r]$ & [7,1] & [7,1] & [8,3] &[2,4] & \\
\hline
$I_r$ & [6,9] & $[7,9]\cup\{2\}$ & [8,2] &[9,3] &\\
 \hline
\end{tabular}
\end{table}

\end{ex}






\begin{rem}
The converse of Theorem \ref{main:theo} holds up to $n=5$. That is, $\sigma\lessdot\pi_{k,5}$ if and only if $\sigma=\overleftarrow{\rho_A}(\pi_{k,n})$ for some $A\in {\mleft(\genfrac..{0pt}{}{[n]}{\ell}\mright)} $ where $\ell=0,\dots,k-1$. In Table \ref{tab:casos faltantes} we show for various values of $k$ and $n$ the number of positroids of rank $k-1$ on $[n]$ that are a quotient of $\pi_{k,n}$ but not characterized via Theorem \ref{main:theo}. In fact, the only instance in which it fails for $n=6$ is when $k=3$. Namely, the positroid $\pi=\pi_{3,6}=456123$ covers 24 positroids in the poset $P_6$ and there are $22=1+6+15$ positroids of the form $\overleftarrow{\rho_A}(\pi)$, where $A\in{\mleft(\genfrac..{0pt}{}{[6]}{\ell}\mright)}$ and $\ell=0,1,2$. The two extra positroids not obtained from this characterization correspond to the permutations $\sigma=652143$ and $\tau=416325$. Notice that $\sigma=\overleftarrow{\rho_{135}}(\pi)$ and $\tau=\overleftarrow{\rho_{246}}(\pi)$.
\end{rem}

\begin{table}[h]
\centering
\begin{tabular}{@{}lllll@{}}
\toprule
n                                         & k                      & $\#$ of quotients       & Characterized by Theorem \ref{main:theo}  & Missing     \\ \midrule
\multicolumn{1}{|l|}{6}                   & \multicolumn{1}{l|}{3} & \multicolumn{1}{l|}{24}  & \multicolumn{1}{l|}{22} & \multicolumn{1}{l|}{2}  \\ \midrule
\multicolumn{1}{|l|}{7}  & \multicolumn{1}{l|}{3} & \multicolumn{1}{l|}{36}  & \multicolumn{1}{l|}{29} & \multicolumn{1}{l|}{7}  \\ \cmidrule(l){2-5} 
\multicolumn{1}{|l|}{}                    & \multicolumn{1}{l|}{4} & \multicolumn{1}{l|}{71}  & \multicolumn{1}{l|}{64} & \multicolumn{1}{l|}{7} \\ \midrule
\multicolumn{1}{|l|}{8}  & \multicolumn{1}{l|}{3} & \multicolumn{1}{l|}{55}  & \multicolumn{1}{l|}{37} & \multicolumn{1}{l|}{18} \\ \cmidrule(l){2-5} 
\multicolumn{1}{|l|}{}                    & \multicolumn{1}{l|}{4} & \multicolumn{1}{l|}{119} & \multicolumn{1}{l|}{93} & \multicolumn{1}{l|}{26} \\ \cmidrule(l){2-5} 
\multicolumn{1}{|l|}{}                    & \multicolumn{1}{l|}{5} & \multicolumn{1}{l|}{179} & \multicolumn{1}{l|}{163} & \multicolumn{1}{l|}{16}\\ \midrule
\multicolumn{1}{|l|}{9}  & \multicolumn{1}{l|}{3} & \multicolumn{1}{l|}{85}  & \multicolumn{1}{l|}{46}  & \multicolumn{1}{l|}{39}\\ \cmidrule(l){2-5} 
\multicolumn{1}{|l|}{}                    & \multicolumn{1}{l|}{4} & \multicolumn{1}{l|}{202} & \multicolumn{1}{l|}{130} & \multicolumn{1}{l|}{72}\\ \cmidrule(l){2-5} 
\multicolumn{1}{|l|}{}                    & \multicolumn{1}{l|}{5} & \multicolumn{1}{l|}{322} & \multicolumn{1}{l|}{256} & \multicolumn{1}{l|}{66}\\ \cmidrule(l){2-5} 
\multicolumn{1}{|l|}{}                    & \multicolumn{1}{l|}{6} & \multicolumn{1}{l|}{412} & \multicolumn{1}{l|}{382} & \multicolumn{1}{l|}{30}\\ \midrule
\multicolumn{1}{|l|}{10} & \multicolumn{1}{l|}{3} & \multicolumn{1}{l|}{133} & \multicolumn{1}{l|}{56} & \multicolumn{1}{l|}{77}\\
\cmidrule(l){2-5} 
\multicolumn{1}{|l|}{}                    & \multicolumn{1}{l|}{4} & \multicolumn{1}{l|}{343} & \multicolumn{1}{l|}{176} & \multicolumn{1}{l|}{167}\\
\cmidrule(l){2-5} 
\multicolumn{1}{|l|}{}                    & \multicolumn{1}{l|}{5} & \multicolumn{1}{l|}{583} & \multicolumn{1}{l|}{386} & \multicolumn{1}{l|}{197}\\
\cmidrule(l){2-5} 
\multicolumn{1}{|l|}{}                    & \multicolumn{1}{l|}{6} & \multicolumn{1}{l|}{773} & \multicolumn{1}{l|}{638} & \multicolumn{1}{l|}{135}\\ \cmidrule(l){2-5} 
\multicolumn{1}{|l|}{}                    & \multicolumn{1}{l|}{7} & \multicolumn{1}{l|}{898} & \multicolumn{1}{l|}{848} & \multicolumn{1}{l|}{50}\\ \bottomrule
\end{tabular}
\caption{Number of quotient positroids of $U_{n,k}$.}
    \label{tab:casos faltantes}
\end{table}

 We highlight that although there are quotients of $U_{k,n}$ that do not satisfy all the conditions in \ref{main:theo}, these quotients also seem to have the same circuit structure that we have presented through cyclic shifts. The following examples illustrate this in the case of uncharacterized uniform positroid quotients and general positroid quotients.

\begin{ex}\label{missingUniformQuotients}
Consider the uniform positroid $U_{4,9}$ indexed by its decorated permutation $\pi_{4,9}=678912345$ and let $A=\{1,3,5,6,8\}=\{1\}\cup\{3\}\cup[5,6]\cup\{8\}$. So in this case $k=4$ and $|A|=5$. We obtain that $\sigma:=\rho_A(\pi)=698214375$. On the other hand, calculating the circuits of $P_\sigma$ we get that $$
\mathcal C_A=\{2,3,4\}\cup\{4,5,6\}\cup\{7,8\}\cup\{1,2,9\}\cup\left\{C\in {\mleft(\genfrac..{0pt}{}{[9]}{4}\mright)} \,:\,[2,4],[4,6],[7,8],[9,1]\not\subset C\right\}.
$$ Making use of SageMath we conclude that $\sigma\lessdot\pi_{4,9}$.
\end{ex}

\begin{ex}\label{missingQuotients}
Consider the positroid $P_\sigma$ indexed by the decorated permutation $\sigma=698214375$ and let $A=\{2,4,7,8,9\}=[2]\cup[4]\cup[7,8,9]$. So in this case $k=3$ and $|A|=5$. We get  that $\tau:=\rho_A(\sigma)=198234576$ and the circuits of $P_\tau$ are $$
\mathcal C_A=\{1\}\cup\{7,8\}\cup\{2,9\}\cup\left\{C\in {\mleft(\genfrac..{0pt}{}{[9]}{3}\mright)} \,:\,[1],[7,8],\{2,9\}\not\subset C\right\}.
$$ Again, using SageMath we see that $\tau\lessdot\sigma$.
\end{ex}


 \begin{coro}\label{dual}
 Let $\tau\in D_{n-k+1,n}$ be a decorated permutation such that $\tau=\overrightarrow{\rho_B}(\pi_{n-k,n})$ for some $B\in {\mleft(\genfrac..{0pt}{}{[n]}{\ell}\mright)} $ where $\ell=0,\dots,k-1$. Then $\pi_{n-k,n}\lessdot\tau$.
\end{coro}

\begin{proof}
Consider the decorated permutation $\pi_{n-k,n}$ corresponding to the uniform matroid $U_{n-k}$. Let $B\in {\mleft(\genfrac..{0pt}{}{[n]}{\ell}\mright)} $ where $\ell=0,\dots,k-1$ and $\tau=\overrightarrow{\rho_B}(\pi_{k,n})$. Now, taking their corresponding positroids, we have that $U_{n-k,n}$ is a quotient of $P_{\tau}$ if and only if $P_{\tau}^*$ is a quotient of $U_{n-k,n}^*$. Since $U_{n-k,n}^*=U_{k,n}$ and $P_{\tau}^*=P_{{\tau}^{-1}}$ (see \cite{ohHal}), we have that $U_{n-k,n}$ is a quotient of $P_{\tau}$ if and only if $P_{{\tau}^{-1}}$ is a quotient of $U_{k,n}$. Now this occurs if and only if ${\tau}^{-1}\leq \pi_{k,n}$. Since $\tau=\overrightarrow{\rho_B}(\pi_{k,n})$, we know that ${\tau}^{-1}=\overleftarrow{\rho_A}(\pi_{n-k,n})$ where $A={\tau}^{-1}(B)$ because of Proposition \ref{prop:inverse and dual}. But notice that $A\in {\mleft(\genfrac..{0pt}{}{[n]}{\ell}\mright)}$ where $\ell=0,\dots,k-1$. Therefore using Theorem \ref{main:theo} we get that ${\tau}^{-1}\leq \pi_{k,n}$. Following back our trail of if and only ifs this means that $U_{n-k,n}$ is a quotient of $P_{\tau}$ and $\pi_{n-k,n}\lessdot\tau$ as we wished.

\end{proof}

In view of Corollary \ref{dual}, we see that $\pi_{n,n}$ covers $r$ positroids if and only if $\pi_{0,n}$ is covered by $r$ positroids. However, since $\pi_{0,n}$ and $\pi_{n,n}$ are the bottom and top elements of $P_n$, respectively, then $r=|D_{1,n}|=|D_{n-1,n}|=2^n-1$ and $|D_{1,n}|=\sum_{\ell=0}^{n-1}{\mleft(\genfrac..{0pt}{}{n}{\ell}\mright)}$. This allows us to conclude that the converse of Theorem \ref{main:theo} also holds for $\pi_{n,n}$. 

As we would like to understand concordance from different points of view, we pose the following conjecture which has been checked for flags of positroids obtained from a single matrix representation.

\begin{conj}
Let $\tau_1, \tau_2\in P_n$ be such that $\tau_1\lessdot\tau_2$. Let $ I_{i}=(I_1^i,\cdots,I_n^i)$ be the Grassmann necklace of $\tau_i$, for $i=1,2$, respectively. Then $I_j^1\subseteq I_j^2$ for each $j\in[n]$.
\end{conj}

\begin{rem}
Containment of Grassmann necklaces does not guarantee quotient property. For example, the necklace $I=(1,3,3,1)$ is contained in $J=(12,23,34,41)$, componentwise. However, the positroid corresponding to $J$ has $\{2,3,4\}$ as a circuit which can not be written as a union of elements in  $\mathcal C_{I}=\{\{2\}, \{4\},\{1,3\}\}$.
\end{rem}

We end by providing a conjecture that summarizes the findings detailed in this paper. This conjecture is based on evidence generated using SageMath.

\begin{conj}
Let $\sigma$ and $\pi$ be positroids on $[n]$ of ranks $k-1$ and $k$, respectively. Then $\sigma\lessdot\pi$ if and only if $\sigma=\overleftarrow{\rho_A}(\pi)$ for some $A\subset[k]$.
\end{conj}

\section{Future work}\label{future}

The main endeavor for the authors' future work is the complete characterization of chains of positroids that result in a flag positroid. As pointed out in the introduction, this may open the door to the notion of realizability for positively oriented flag matroids, in the vein of the work of \cite{ARWOriented}.

A parallel approach to the concordance problem can be taken via positroid polytopes. That is, study chains of positroids that form a flag and inequalities of the corresponding flag polytope. That is, can we characterize flag positroids via flag matroid polytopes? 

Finally, the path took here toward the concordance problem has unveiled the poset of concordance which deserves to be explored further. Some interesting questions in this direction are:
\begin{itemize}
    \item[(a)] \emph{What is the Mobius function of the poset $P_n$?} Up to $n=4$ the first values of $\mu(P_n)$ are $1,-1,2,-9,92$.
    \item[(b)] \emph{Is there an ER-labelling of $P_n$?} A candidate for such labelling is labelling the edge of the covering $\tau\lessdot\sigma$ by the set frozen when passing from $\sigma$ to $\tau$. Answering this question may give a Whitney dual for this poset, in the sense of~\cite{GonzalezHallma}.
\end{itemize}

\section{Acknowledgements}
C. Benedetti is grateful with the Faculty of Science of Universidad de los Andes and the grant FAPA. A. Chavez thanks the National Science Foundation under Award No. 1802986.


\bibliographystyle{alpha}
\bibliography{Quotients_of_uniform_positroids}

\end{document}